\documentclass[12pt]{amsart}
\usepackage{amsfonts}
\usepackage{amssymb}
\usepackage{amsmath}
\usepackage{amsthm,tikz}
\usepackage{comment}

\newtheorem{proposition}{Proposition}
\DeclareMathOperator{\vol}{\boldmath{v}}
\DeclareMathOperator{\str}{str}
\DeclareMathOperator{\tr}{tr}

\DeclareMathOperator{\im}{im}
\newcommand{\cl}{\mathit{Cl}}
\newcommand{\dotp}{\boldsymbol{\cdot}}
\title{Riemannian geometry without indices.}
\date{\today}
\author{V.V.Fock}
\author{P.Goussard}
\address{V.F. : IRMA, Universit\'e de Strasbourg, 7 rue Ren\'e Descartes 67084 Strasbourg France}
\email{fock@math.unistra.fr}
\address{P.G. : ??}
\email{pierre.goussard@ac-mayotte.fr}
\begin{document}

\maketitle

\begin{abstract}
We suggest an index-free formalism allowing to simplify many computations in Riemann geometry. The main ingredients are forms with values in a Clifford algebra and an action of the group $\mathfrak{sl}_2\times \mathfrak{sl}_2$ on such forms.

\end{abstract}

\section{Introduction} Working with Levi-Civita connection, curvature, Weyl and Ricci tensors and in particular deriving Einstein equation out of the Hilbert action is a painful struggle with indices (see for example \cite{DNF}). In the present note we suggest a version of Cartan-Palatini formalism \cite{P}\cite{C}\cite{W} allowing to make most of the computations without indices at all. The main objects for this approach are Clifford forms: differential forms with values in the Clifford algebra $\cl(\mathrm{V})$ generated by a vector space $\mathrm{V}$ provided with a nondegenerate quadratic form $\eta$. Since Clifford algebra is graded as a vector space, such differential forms are bigraded. The corresponding graded components are denoted by $\Omega^{pq}=\Omega^q(M,\cl^p(\mathrm{V}))$. The elements of such components are called $(p,q)$-forms.

Clifford forms are known to be useful to work with spinors, but here we would like to emphasize that they simplify computations even if no spinors are under consideration.

The Riemann metric on a manifold is replaced by a vielbein $\theta$  which is an (1,1)-form and and a spin connection $\omega$, which is a (2,1)-form. The torsion can be written as a (1,2)-form $t=d\theta+\omega\theta+\theta\omega$ and the curvature as a (2,2)-form $R=d\omega+\omega^2$. The Hilbert action takes in these variables the form
$$S=\str\int\theta^{n-2}(d\omega + \omega^2),$$ 
where $\str$ denotes the super-trace on the Clifford algebra. The Hilbert action is invariant under the gauge transformation $\theta\mapsto g^{-1}\theta g$, $\omega \mapsto g^{-1}dg + g^{-1}dg$ with $g$ a function on the manifold with values in the spin group considered as a subgroup of the Clifford algebra.

In presence of the vielbein $\theta$ one can define an action of the Lie algebra $\mathfrak{sl}_2\times \mathfrak{sl}_2$ on the space of Clifford forms analogous to the Lefschetz algebra $\mathfrak{sl}_2$ in topology of K\"ahler manifolds. One can define two Hodge star operator on Clifford forms acting in the fiber and in the base, respectively. The algebra $\mathfrak{sl}_2\times \mathfrak{sl}_2$ is generated by external multiplication by the form $\theta$ as well as by its conjugates by the Hodge stars.
This algebra simplifies working with the suggested formalism. In particular Einstein equation together with the first Bianchi identity in four dimensions amount to the invariance of the curvature tensor under this algebra. Which in its turn implies self-duality and symmetry of the curvature.

In the end of this paper the construction is generalized to complex manifolds where the space of Clifford forms is $\mathbb{Z}^4$ graded and the algebra $\mathfrak{sl}_2\times \mathfrak{sl}_2$ is replaced by the affine Lie algebra $\widehat{\mathfrak{sl}}_4$.

\section{Grassmann and Clifford algebras.}
In this section we recall basic notions about Grassmann (exterior) and Clifford algebras and relations between them. Then we define the algebra of Clifford forms and the action of the algebra $\mathfrak{sl}_2\times\mathfrak{sl}_2$ thereon.

A Clifford algebra $\cl(\mathrm{V})$ of a vector space $\mathrm{V}$ provided with a nondegenerate quadratic form $\eta\in \mathrm{V}^*\otimes \mathrm{V}^*$ is an algebra generated by $\mathrm{V}$ with relations $vv'+v'v=\eta(v,v')$ for any $v,v'\in \mathrm{V}$. A Grassmann algebra $\Lambda(\mathrm{V})$ of a vector space $\mathrm{V}$ is an algebra generated by $\mathrm{V}$ with relations $vv'+v'v=0$ for any $v,v'\in \mathrm{V}$. The Grassmann algebra is graded with $\deg v=1$ for any $v\in \mathrm{V}$. The corresponding graded components will be denoted by $\Lambda^{i}(\mathrm{V})$.

The Grassmann algebra $\Lambda(\mathrm{V})$ faithfully acts on the Clifford algebra  $\cl(\mathrm{V})$ by $v\wedge a\mapsto \frac12(va+(-1)^{\deg a}av)$ for any $v\in \mathrm{V}\subset \Lambda(\mathrm{V})$ and any homogeneous $a \in \cl(\mathrm{V})$. This action allows to define an isomorphism between Grassmann and Clifford algebras as modules over the Grassmann algebra. The isomorphism is fixed by the requirement that $1\mapsto 1$. Denote by $\cl^i(\mathrm{V})$ the subspaces of the Clifford algebra corresponding to the homogeneous components $\Lambda^i(\mathrm{V})$ of the Grassmann algebra with respect to this isomorphism and $\Pi^i:\cl(\mathrm{V})\to \cl^i(\mathrm{V})$ the projection on the corresponding component. 

The component $\cl^n(\mathrm{V})$ is of dimension one and contains two elements $v$ such that $v^2=1$. Choosing one of these two elements fixes the orientation of the space $V$. There is a second action of $\Lambda(V)$ given by $v\vdash a\mapsto \frac12(va-(-1)^{\deg a}av)$ and the second isomorphism corresponding of Grassmann and Clifford algebras as modules sending $1$ to $v$. Define the Hodge star operator as the automorphism $*$ of the Clifford algebra as a vector space intertwining these two isomorphisms, namely such that $v\wedge (*a)=*(v\vdash a)$. Obviously $*:\cl^i(\mathrm{V})\to \cl^{n-i}(\mathrm{V})$. The identification of $\cl(\mathrm{V})$ and $\Lambda(\mathrm{V})$ allows to consider $*$ also as an automorphism of the Grassmann algebra (as a vector space). The actions $v\wedge$ and $v\vdash$ on the Clifford algebra transferred to the Grassmann algebra are just the usual external and internal multiplication by $v$.

Observe that though the isomorphism of $\cl(\mathrm{V})$ and $\Lambda(\mathrm{V})$ is not an algebra isomorphism one can express the external and the internal product in terms of Clifford product and the projections:
\begin{align*}
x\wedge y= \Pi^{\deg x + \deg y}xy, &&
x\vdash y= \Pi^{\deg y - \deg x}xy.
\end{align*}

Given a basis in the space $V$ the Grassmann algebra can be identified with the algebra of polynomials of $n$ Grassmann variables $\xi_1,\ldots,\xi_n$ and the Hodge $*$ operator to a Grassmann Fourier transformation, see e.g. \cite{L}, (here and below we assume the Einstein summation convention over repeated indices):
$$*f(\xi_1,\ldots,\xi_n)=\int e^{\displaystyle \eta^{ij}\xi_i\zeta_j}f(\zeta_1,\ldots,\zeta_n)d\zeta_n\cdots d\zeta_1.$$ One can easily verify the properties analogous to the usual Fourier transformation:
$$\frac{\partial}{\partial \xi_i}*f=*(\eta^{ij}\xi_jf),\ \xi_i*f=*(\eta_{ij}\frac{\partial}{\partial \xi_j}f)$$
Here $\eta_{ij}=\eta(\xi_i,\xi_j)$ is the scalar product matrix and $\eta^{ij}$ is its inverse.

\subsection{Action of the Lie algebra $\mathfrak{sl_2}\times \mathfrak{sl_2}$ on Clifford forms.} Consider now the algebra $\Omega^{\dotp\dotp}=\cl(V)\otimes\Omega(M)$ of forms on an $n$-dimensional manifold $M$ with values in the Clifford algebra. Let the vielbein $\theta\in \Omega^{11}$ defines an isomorphism between the tangent bundle to $M$ and a trivial bundle with fiber $\mathrm{V}$. One can define two Hodge $*$ operators $*_1:\Omega^{pq}\to \Omega^{(n-p)q}$ and $*_2:\Omega^{pq}\to \Omega^{p(n-q)}$ acting on the Clifford algebra and on forms, respectively. Define an operator $E:\Omega^{pq}\to\Omega^{p+1,q+1}$ by $2Ex=\theta x+(-1)^{p+q}x\theta$ and the operators $E'=*_1^{-1}E*_1$, $F'=*_2^{-1}E*_2$, $F=*_1^{-1}*_2^{-1}E*_2*_1$, $H=p+q-n$ and $H'=p-q$.

These operators satisfy the following properties:
\begin{proposition}\label{sl2}
	\noindent 1. The operators $E,F,H, E',F',H'$ satisfy the relations of the Lie algebra $\mathfrak{sl}_2\times\mathfrak{sl}_2$, namely
	$$[H,E]=2E, [H,F]=-2F, [H',E']=2E', [H',F']=-2F'$$
	and all other pairs of generators commute.\vspace{2mm}
	
	\noindent 2. $\deg E = (1,1)$, $\deg F = (-1,-1)$, $\deg E'= (-1,1)$, $\deg F' = (1,-1)$, , $\deg H = \deg H' =(0,0)$.\vspace{2mm}
	
	\noindent 3. $2E'x=\theta x-(-1)^{p+q}x\theta.$\vspace{2mm}
	
	\noindent 4. $
	\exp(\frac{\pi}{2}(E'-F'))f(\xi_1,\ldots,\xi_n,\xi^1,\ldots,\xi^n)=\\~~~~~~~=f(\eta_{1j}\xi^j\,\ldots,\eta_{nj}\xi^j,\eta^{1j}\xi_j\,\ldots,\eta^{nj}\xi_j).$\vspace{2mm}
	
	\noindent 5.
	$\exp(\frac{\pi}{2}(E-F))=*_1*_2$
	
\end{proposition}

\begin{proof}
The generators act on Clifford forms without any differentiation so we need to verify the proposition at one point of the manifold. At one point $m\in M$ a Clifford form takes value in $\cl(\mathrm{V})\otimes \Lambda(T_mM)$ which, using the isomorphism defined by $\theta$, can be identified with $\cl(\mathrm{V})\otimes\Lambda(\mathrm{V}^*)$. 

To make explicit computation one can further identify this algebra with the algebra $\Lambda(\mathrm{V})\otimes\Lambda(\mathrm{V}^*)$. Given a basis in $\mathrm{V}$ the algebras $\Lambda(\mathrm{V})$ and $\Lambda(\mathrm{V}^*)$ can be identified with the space of functions of $n$ Grassmann variables $\xi_i$ and $\xi^i$, respectively, and the vielbein writes down as $\theta= \xi_i\otimes \xi^i$. The following computation in $\cl(\mathrm{V})\otimes\Lambda(\mathrm{V}^*)$
\begin{equation}\label{Ex}
\begin{split}
2Ex=&\theta x+(-1)^{p+q}x\theta=\xi_ia\otimes \xi^i\alpha+(-1)^{p+q} a\xi_i\otimes \alpha\xi^i=\\ =&\xi_ia\otimes \xi^i\alpha+(-1)^p a\xi_i\otimes \xi^i\alpha=\xi_i\wedge a\otimes \xi^i\alpha,
\end{split}
\end{equation}
where $x=a\otimes\alpha$ with $a\in \cl^p(\mathrm{V})$ and $\alpha \in \Omega^q(M)$, 
shows that the operators $E$ and $E'$ can be written in $\Lambda(\mathrm{V})\otimes\Lambda(\mathrm{V}^*)$ just as multiplication operator by $\xi_i\otimes \xi^i$. Therefore all operators can be re written as
\begin{align*} E&=\xi_i\otimes \xi^i,& F&= \frac{\partial}{\partial \xi_i}\otimes \frac{\partial}{\partial \xi^i} & H&=\,\xi_i\frac{\partial}{\partial \xi_i}\otimes 1+1\otimes \frac{\partial}{\partial \xi^i}\xi^i\\ E'&=\eta_{ij} \frac{\partial}{\partial \xi_i}\otimes \xi^j, &  F'&=\xi_i\otimes \eta^{ij}\frac{\partial}{\partial \xi^j}, & H'&=\,1\otimes \xi^i\frac{\partial}{\partial \xi^i}-\xi_i\frac{\partial}{\partial \xi_i}\otimes 1.
\end{align*}
The property 1 can be verified by a simple straightforward computation. The property 2 is obvious. The property 3 follows from the computation similar to (\ref{Ex}).
\begin{align*}
2E'x&=\theta x-(-1)^{p+q}x\theta=\xi_ia\otimes \xi^i\alpha-(-1)^{p+q} a\xi_i\otimes \alpha\xi^i=\\&=\xi_ia\otimes \xi^i\alpha-(-1)^p a\xi_i\otimes \xi^i\alpha=\xi_i\vdash a\otimes \xi^i\alpha=\eta_{ij}\frac{\partial}{\partial \xi_j}a\otimes \xi^i\alpha.
\end{align*}

The properties 5 and 6 follow from the fact that the operators in the left hand side conjugate the generators exactly as the operation in the left hand side does.
	
The properties 2,3,5 and 6 follow from these expressions immediately. The property 1 also can be verified by a direct computation. The property 4 is a standard statement about finite-dimensional representations of the algebra $\mathfrak{sl}_2$. 
\end{proof}

Other obvious but useful relations are:

\noindent 7. $\theta x = (E+E')x,\ x \theta = (-1)^{p+q}(E-E')x$,\vspace{2mm}

\noindent 8. $E^k x= \Pi^{p+k} \theta^kx= (-1)^{k(p+q)}\Pi^{p+k} x\theta^k$,\vspace{2mm}

\noindent 9. $E^k$ acting on the weight space with weight (eigenvalue of $H$) equal to $l$ is an injection if $k+l\leqslant 0$ and a surjection if $k+l\geqslant 0$.

\subsection{Supertrace.} A \textit{supertrace} of a Clifford algebra is a linear function on $\cl(\mathrm{V})$ satisfying the identity $\str(ab)=(-1)^{\deg a\deg b}\str(ba)$ and normalized by the condition that $\str \vol=1$. If $n$ is even then the Clifford algebra has a unique representation and the supertrace is given by $\str a=\tr{a\vol}$, where $\tr$ is the usual trace in this representation. For $n$ odd the space $\mathrm{V}$ can be embedded into a space $\mathrm{V}^+$ of dimension $n+1$ with one extra generator $\xi^0$ orthogonal to $\mathrm{V}$. The supertrace is given by $\str a=\tr{a\xi^0\vol}$ with the trace taken in the representation of $\cl(\mathrm{V}^+)$. The supertrace vanishes on $\cl^i(\mathrm{V})$ for $i\neq n$. 

In particular for two forms $\alpha\in \Omega^{pq}$ and $\beta\in \Omega^{p'q'}$ we have
$$ \str \alpha\beta = (-1)^{pp'+qq'}\str \beta \alpha.$$

\subsection{The action of the group $Pin$.} The invertible elements of $g\in \cl(\mathrm{V})$ such that $g^{-1}\cl^1(\mathrm{V})g\subset \cl^1(\mathrm{V})$ form a Lie group denoted by $Pin(\mathrm{V})$. The Lie algebra of this group is $\cl^2(\mathrm{V})$ with respect to the commutator. It is isomorphic to the Lie algebra $\mathfrak{so}(\mathrm{V})$ and acts on the whole Clifford algebra preserving degree. The action of the group $Pin(\mathrm{V})$ obviously commutes with the action of the algebra $\mathfrak{sl_2}\times\mathfrak{sl_2} $.

\section{Use of the formalism}

Let $M$ be an $n$-dimensional manifold and $\mathrm{V}$ an $n$-dimensional vector space with a nondegenerate quadratic form $\eta$. Let we are given two Clifford algebra valued 1-forms $\theta\in \Omega^{11}$ and $\omega\in \Omega^{21}$ such that the form $\theta$ is nondegenerate in the sense that it induces an isomorphism at every point of $M$ between the tangent and a trivial bundle with fiber $\mathrm{V}$ by the rule $X\mapsto i_X\theta$. These data define a Riemannian or pseudo-Riemannian metric on $M$ and a metric preserving connection in the tangent bundle in it. Indeed, the form $\omega$ defines a connection in the trivial bundle with fiber $\mathrm{V}$ by the formula $\nabla e = de + \omega e - e\omega$, where $e\in \Omega^{10}$. The isomorphism $\theta$ induces the metric and the connection on $TM$.

\subsection{Gauge group action.}
A pair $\theta,\omega$ defines the same metric and connection as a pair $\theta',\omega'$ if and only if 
there exists a gauge transformation relating them, namely if there exists a function $g$ on $M$ with values in $Pin(\mathrm{V})\subset \cl(\mathrm{V})$ such that $\theta'=g^{-1}\theta g,\ \omega'= g^{-1}\omega g+g^{-1}d g$.

Conversely, any metric and a connection in a tangent bundle preserving it can be represented in this way if and only if the tangent bundle is trivializable. In particular it always can be done locally and glued together by gauge transformations.

\subsection{Curvature.}
The connection $\nabla$ can be extended in a standard way to $\Omega^{pq}$ by the formula $\nabla e \alpha = (\nabla e) \alpha + ed\alpha$ for $e\in \Omega^{p0}$ and $\alpha \in \Omega^{0q}$. It can be therefore written as $\nabla x = dx + \omega x -(-1)^q x\omega$ for any $x\in \Omega^{pq}$. This allows to compute $\nabla^2 x = (d\omega+\omega^2)x-x(d\omega+\omega^2)$ and thus the curvature takes the form $$R=d\omega+\omega^2\in \Omega^{22}.$$

\subsection{Torsion.} The torsion can be defined as an element $t\in \Omega^{12}$ such that for any two vector fields $X$ and $Y$ on $M$ we have $i_Yi_Xt = i_X\nabla i_Y\theta-i_Y\nabla i_X\theta-i_{[X,Y]}\theta$. Taking into account the identity $i_Xi_Yd\alpha=i_Ydi_X\alpha-i_Xdi_Y\alpha+i_{[X,Y]}\alpha$ valid for any 1-form $\alpha$ one can easily verify that $$t=d\theta+\omega\theta+\theta\omega=d\theta+2E'\omega.$$ The operator $E':\Omega^{21}\to \Omega^{12}$ is an isomorphism since $H'|_{\Omega^{21}}=-1$ (property 9). It implies that given $\theta$ the torsion determines the form $\omega$ unambiguously. In particular for a given $\theta$ there exists a unique form $\omega$ such that the torsion vanishes. A pair $\theta, \omega$ for which the torsion vanishes is called \textit{torsionless}.

\subsection{Bianchi identities.} Computing the covariant derivative of the torsion one gets $$\nabla t= dt+\omega t-t\omega=R\theta-\theta R=-2E'R.$$ In particular this identity implies that if the torsion vanishes we have $E'R=0$. Since $H'R=0$ we have $F'R=0$ and therefore  the curvature form is invariant with respect to the second $\mathfrak{sl_2}$ algebra and therefore by the property 5 the form $R$ considered as a function with values in $\Lambda^2(\mathrm{V})\otimes \Lambda^2(\mathrm{V})$ is symmetric. The observation that the Bianchi identity is a consequence of an $\mathfrak{sl_2}$ invariance belongs to P.\v{S}evera \cite{S}.

The second Bianchi identity $\nabla R=0$ can be verified by a short direct calculation.

\subsection{Conformal transformations and the Weyl tensor.} Let $(\theta,\omega)$ be a pair of Clifford forms with vanishing torsion. Let $\tilde{\theta}=e^\phi\theta$ be a conformal transformation of the form $\theta$. This transformation corresponds to a conformal transformation of the metric. Compute now the induced transformation of the form $\omega$ and of the curvature $R$.

\begin{proposition}
Let $\theta,\omega$ be a torsionless pair of Clifford forms, $R$ be its curvature and $\phi$ a function on $M$. Than the pair
$$
\tilde{\theta}=e^\phi\theta,\ \tilde{\omega}=\omega+\theta \varepsilon-\varepsilon\, \theta,
$$
is also torsionless with curvature
$$\tilde{R}=R-\theta \rho-\rho\theta=R-2E\rho$$
where 
$$ \varepsilon = \frac14 F'd\phi\in \Omega^{10}$$
and
$$
\rho=d\varepsilon +\omega\varepsilon-\varepsilon\omega+\varepsilon\theta\varepsilon\in \Omega^{11}.
$$

\end{proposition}
It implies that the image of the curvature $R$ in the cokernel of $E$ in $\Omega^{22}$ is invariant under conformal transformations. In dimension four or higher this cokernel is nontrivial. Taking into account that for any $k$ we have $\Lambda^{\dotp\dotp}=\im E\oplus \ker F$. The projection of $R$ onto $\ker F\cap\Lambda^{22}$ along $E(\Omega^{11})$ is called the Weyl tensor and is denoted by $W$.

\begin{proof} Observe first that $E'F'\alpha =\alpha$ for $\alpha\in\Omega^{01}$. Therefore $E'F'd\phi=\frac12\theta (F'd\phi)+\frac12(F'd\phi)\theta = d\phi$. The torsion and the curvature can be computed directly:

\begin{align*}
&e^{-\phi}(d\tilde{\theta}+\tilde{\theta}\tilde{\omega}+\tilde{\omega}\tilde{\theta})=e^{-\phi}d(e^{\phi}\theta)+\theta\omega +\omega \theta+\\&+\theta(\theta \varepsilon-\varepsilon \theta)+(\theta \varepsilon-\varepsilon \theta)\theta=d\theta +\theta\omega +\omega \theta+d\phi\theta+\\ &+\theta(\theta \varepsilon+\varepsilon \theta)-(\theta \varepsilon+\varepsilon \theta)\theta=d\phi\,\theta+\frac12\theta d\phi-\frac12 d\phi\,\theta=0.
\end{align*}
and 

\begin{align*}
\tilde{R}=&d\tilde{\omega}+\tilde{\omega}^2= R+d(\theta\varepsilon-\varepsilon\theta)+\omega(\theta\varepsilon-\varepsilon\theta)+(\theta\varepsilon-\varepsilon\theta)\omega+(\theta\varepsilon-\varepsilon\theta)^2=\\=&-(\omega\theta+\theta\omega) \varepsilon+\varepsilon(\omega\theta+\theta\omega) -\theta d\varepsilon-d\varepsilon\theta+\omega(\theta\varepsilon-\varepsilon\theta)+(\theta\varepsilon-\varepsilon\theta)\omega+\\+&(\theta\varepsilon-\varepsilon\theta)^2=R-\theta(d\varepsilon +\omega\varepsilon-\varepsilon\omega+\varepsilon\theta\varepsilon)-(d\varepsilon +\omega\varepsilon-\varepsilon\omega+\varepsilon\theta\varepsilon)\theta.
\end{align*}
\end{proof}

\subsection{Hilbert action.}

On the space of pairs $\theta,\omega$ define the Hilbert functional:
$$S(\theta, \omega)=\str \int_M \theta^{n-2}(d\omega+\omega^2)=\str\int_M E^{n-2}R.$$

\begin{proposition} The Hilbert functional is gauge invariant.
The variation of this functional is
\begin{align*}\frac{\delta S}{\delta \theta} =&(-1)^{(n^2+n)/2}(n-2)E^{n-3}(d\omega+\omega^2)\\
\frac{\delta S}{\delta \omega}=& (-1)^{(n^2-n)/2}E^{n-3}(d\theta+\theta\omega+\omega\theta).
\end{align*}
The induced closed  2-form on the space $\Omega^{11}(\partial M)\oplus \Omega^{n-1,n-2}(\partial M)$ is
$$w=\pm(n-2)\str\int_{\partial M}\delta\theta E^{n-3}\delta\omega$$
\end{proposition}

Since $E^{n-3}:\Omega^{12}\to\Omega^{n-2,n-1}$ is an isomorphism it implies that the vanishing of $\frac{\delta S}{\delta \omega}$ is equivalent to the vanishing of the torsion $t=d\theta+\omega\theta+\theta\omega$. Similarly since $\ker E^{n-3}:\Omega^{22}\to\Omega^{n-1,n-1}=\ker F:\Omega^{22}\to \Omega^{11}$, the vanishing of $\frac{\delta S}{\delta \theta}$ is equivalent to the vanishing of the Ricci tensor $r=FR$.
\begin{proof}
Computing the variation of the Hilbert functional 
$$\delta S=\str\int_{\partial M} \theta^{n-2}\delta\omega + \str\int_M \frac{\delta S}{\delta \theta}\delta \theta + \str\int_M \frac{\delta S}{\delta \omega}\delta \omega$$
one gets
$$\frac{\delta S}{\delta \theta} =\Pi^{n-1}\left(\sum_{i=0}^{n-3-i}\theta^{n-3-i}(d\omega+\omega^2)\theta^i\right)=(n-2)\Pi^{n-1}\left((d\omega+\omega^2)\theta^{n-3}\right).
$$
On the other hand
\begin{align*}
\frac{\delta S}{\delta \omega}=&\Pi^{n-2}\left((-1)^{n-1}d\theta^{n-2}+ \theta^{n-2}\omega+(-1)^{n-1}\omega\theta^{n-2}\right)=\\=&\Pi^{n-2}(\theta^{n-3}(d\theta+\omega\theta+\theta\omega)).
\end{align*}
Using the property 8  one gets the desired answer.

The expression for the 2-form is given by the variation of the boundary term.

\begin{align*}
w=& \delta\str\int_{\partial M}\theta^{n-2}\delta\omega=\str\int_{\partial M}\sum_{i+j=n-3}\theta^i\delta\theta\theta^j\delta\omega=\\=&\str\int_{\partial M}\sum_{i+j=n-3}(-1)^i\delta\theta\Pi^{n-1}(\theta^j\delta\omega\theta^i)=\\=&(n-2)(-1)^{n-1}\str\int_{\partial M}\delta\theta\Pi^{n-1}(\delta\omega\theta^{n-3})=\\=&(n-2)(-1)^{n-1}\str\int_{\partial M}\delta\theta E^{n-3}\delta\omega.
\end{align*}
\end{proof}
\begin{center}
\begin{tikzpicture}[]
\draw[step=1cm,gray,very thin,rotate around={45:(0,0)}](0,0) grid (4,4);
\draw[gray,very thin,->,rotate around={45:(0,0)}](0,0)--(4.5,0);
\draw[gray,very thin,->,rotate around={45:(0,0)}](0,0)--(0,4.5);
\node at (0,2.8) {$R,W$};
\node at (0,1.4) {$\theta$,$r$};
\node at (-0.7,2.1) {$\omega$};
\node at (0.7,2.1) {$t$};
\begin{scope}[xshift=5.5cm,yshift=2.8cm]
\draw[->] (0,0)--(1.4,0);\node[right] at (1.4,0) {$E'$};
\draw[->] (0,0)--(0,1.4);\node[above] at (0,1.4) {$E$};
\draw[->] (0,0)--(-1.4,0);\node[left] at (-1.4,0) {$F'$};
\draw[->] (0,0)--(0,-1.4);\node[below] at (0,-1.4) {$F$};
\end{scope}
\node at (3.3,3.3) {$q$};
\node at (-3.3,3.3) {$p$};
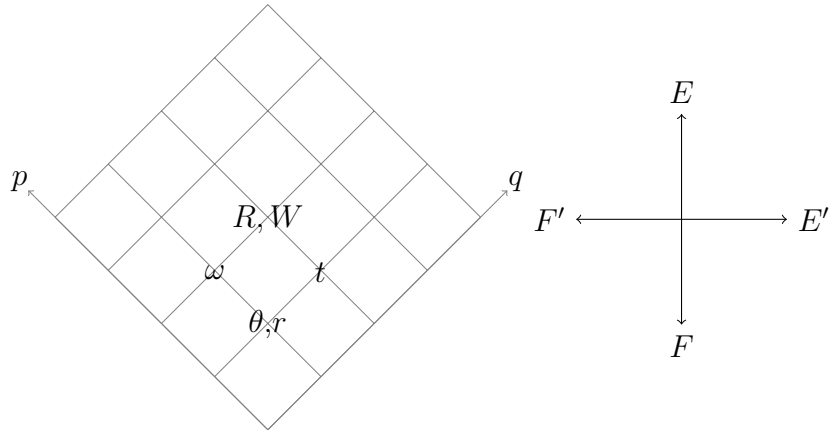
\node at (2,-1) {Zoo of Clifford forms and the action of $\mathfrak{sl}_2\times\mathfrak{sl}_2$ $(n=4)$};
\end{tikzpicture}
\end{center}

\subsection{Special case $n=4$.}
The dimension 4 of the manifold $M$ is special in particular since in this case $HR=0$ and the equation of motion (Einstein equation) amounts to $ER=0$. Together with the Bianchi identity $E'R=0$ it implies that the curvature $R$ is invariant under the action of both $\mathfrak{sl}_2$ algebras. In this dimension the action of the operators $*_1$ and $*_2$ are involutions preserving the subspace $\Omega^{22}$ and therefore the curvature can be decomposed as $R=R_{++}+R_{+-}+R_{-+}+R_{--}$ according to the eigenvalues of the operators $*_1$ and $*_2$. By the property 6 the curvature $R$ is invariant under the product $*_1*_2$ and therefore it is self-dual, i.e., $R_{+-}=R_{-+}=0$, see \cite{AHS}. The components $R_{++}$ and $R_{--}$ are symmetric and traceless, since $E^2R=E'^2R=0$. Therefore the curvature $R$ takes value in a 10-dimensional space, which is possible to deduce also from the analysis of the decomposition of $\Omega^{\dotp\dotp}$ into irreducible components w.r.t to the $\mathfrak{sl}_2\times\mathfrak{sl}_2$ action. Namely the decomposition into irreducibles is given by
\begin{align*} \Omega^{\dotp\dotp}=W_{40}\oplus W_{04}\oplus \mathbb{C}^4\otimes W_{13}\oplus \mathbb{C}^4\otimes W_{31}\oplus \mathbb{C}^6\otimes W_{22}\oplus \\\oplus\mathbb{C}^9\otimes W_{20}\oplus\mathbb{C}^9\otimes W_{02}\oplus \mathbb{C}^{16}\otimes W_{11}\oplus\mathbb{C}^{10}\otimes W_{00} ,
\end{align*}
where $W_{ij}$ is an irreducible representation with the highest weight $(i,j)$ of dimension $(i+1)(j+1)$.

\section{Complex structure} In this section we generalize the construction for the case where the space $\mathrm{V}$  is provided with a complex structure $J$ compatible with the metric. In order to simplify notation we rename this space, vielbein, connection, torsion, curvature etc. into  $\mathrm{V}^{\mathbb{R}}$, $\theta^{\mathbb{R}}$ $t^{\mathbb{R}}$, $R^{\mathbb{R}}$ etc.

Recall that a complex structure of on $\mathrm{V}^\mathbb{R}$ is an operator $J:\mathrm{V}^{\mathbb{R}}\to \mathrm{V}^{\mathbb{R}}$ with $J^2=-1$ and preserving the metric $\eta$, namely,  such that $\eta(v_1,v_2)=\eta(Jv_1,Jv_2)$. Such operator gives a decomposition of the complexification $\mathrm{V}^\mathbb{C}$ of the space $\mathrm{V}^{\mathbb{R}}$ into a sum of isotropic eigenspaces of this operator $\mathrm{V}^\mathbb{C}=\mathrm{V}\oplus \bar{\mathrm{V}}$ and therefore the algebra $\Lambda(\mathrm{V}^{\mathbb{R}})\otimes \Lambda((\mathrm{V}^{\mathbb{R}})^*) = \Lambda(\mathrm{V}+\bar{\mathrm{V}})\otimes \Lambda(\mathrm{V}^*+\bar{\mathrm{V}}^*)$ acquires a $\mathbb{Z}^4$ grading. The corresponding graded components are denoted by $\Lambda^{p\bar{p}q\bar{q}}$. 

Given a complex structure in $\mathrm{V}^\mathbb{R}$ the vielbein $\theta^{\mathbb{R}}$ induces a complex structure on every tangent space to the manifold $M$ and thus an almost complex structure on $M$. The space of Clifford forms on $M$ is therefore graded by $\mathbb Z_4$. We denote by $\Omega^{\dotp\dotp\dotp\dotp}$ the space of $\cl(\mathrm{V}^\mathbb{C})$-valued forms on $M$. The corresponding graded components are denoted by $\Omega^{p\bar{p}q\bar{q}}$ and the whole graded algebra by $\Omega^{\dotp\dotp\dotp\dotp}$.  With respect to this complex structure we have $\theta\in \Omega^{1010}$ and $\bar{\theta}\in \Omega^{0101}$.  The complex structure is integrable if $d\theta \in\Omega^{1020}\oplus \Omega^{1011}$  (and the component in $\Omega^{1002}$ vanishes). In this case one can define a decomposition of the differential $d=\partial+\bar{\partial}$ with the property $\partial:\Omega^{p\bar{p}q\bar{q}}\to \Omega^{p\bar{p}(q+1)\bar{q}}$ and $\bar{\partial}:\Omega^{p\bar{p}q\bar{q}}\to \Omega^{p\bar{p}q(\bar{q}+1)}$.

\subsection{Action of the algebra $\widehat{\mathfrak{sl}}_4$ on complex Clifford forms.}
Define the operators 

\begin{align*} 2E_0x&=\theta x+(-1)^{(p+\bar{p}+q+\bar{q})}x\theta,&
2F_1x&=\theta x-(-1)^{(p+\bar{p}+q+\bar{q})}x\theta\\
2F_2x&=\bar{\theta} x+(-1)^{(p+\bar{p}+q+\bar{q})}x\bar{\theta}, &
2E_3&=\bar{\theta} x-(-1)^{(p+\bar{p}+q+\bar{q})}x\bar{\theta}
\end{align*}

Observe that $E_3=*_1^{-1}F_2*_1$ and $F_1=*_1^{-1}E_0*_1$. Define also $E_1=*_2^{-1}F_2*_2$, $E_2=*_1^{-1}*_2^{-1}E_0*_2*_1$, $F_0=*_1^{-1}*_2^{-1}F_2*_2*_1$, $H_0=n-q''-p$, $H_1=q''-p''$, $H_3=p'-q'$ and $H_2=n-q'-p''$. These operators have degrees
\begin{align*} \deg{E_0}&=(1,0,1,0),& \deg{E_1}&=(0,1,-1,0),& \deg{E_2}&=(0,-1,0,-1),\\ \deg{E_3}&=(-1,0,0,1), &
\deg{F_0}&=(-1,0,-1,0),& \deg{F_1}&=(0,-1,1,0), \\ \deg{F_2}&=(0,1,0,1),&\deg{F_3}&=(1,0,0,-1), &\deg H_i&=(0,0,0,0).
\end{align*}

\begin{proposition}
The operators $\{E_i,F_i,H_i\}$  generate the action of the affine Lie algebra $\widehat{\mathfrak{sl}_4}$ on $\Omega^{\dotp\dotp\dotp\dotp}$
\end{proposition}

\begin{proof}
We need to check the following commutation relations:

\begin{gather}
[E_i,F_i]=H_i,\quad [H_i,E_i]=2E_i,\quad [H_i,F_i]=-2F_i,\label{EFH}\\
[E_i,[E_i,E_{i\pm 1}]]=0,\ [F_i,[F_i,F_{i\pm 1}]]=0,\label{Serre}\\
[E_i,F_j]=0 \mbox{ for } i\neq j,\label{EF}\\
[E_i,E_{i+2}]=0\label{EE0}
\end{gather}

Since the operators are not differential it is sufficient to check the relation in every point of the manifold $M$ in the algebra $\Lambda(\mathrm{V}^{\mathbb{R}})\otimes \Lambda((\mathrm{V}^{\mathbb{R}})^*) = \Lambda(\mathrm{V}+\bar{\mathrm{V}})\otimes \Lambda(\mathrm{V}^*+\bar{\mathrm{V}}^*)$.

Introducing a base $\xi_i$ of the space $\mathrm{V}'$ one can identify the algebra $\Lambda(\mathrm{V}+\bar{\mathrm{V}})$ with polynomials of $2n$ Grassmann variables $\xi_k$, $\bar{\xi}_k$ and the algebra $\Lambda(\mathrm{V}^*+\bar{\mathrm{V}}^*)$with polynomials of $2n$ Grassmann variables  $\xi^k$, $\bar{\xi}^k$. Taking into account that 
\begin{align*}
*^{-1}\frac{\partial}{\partial \xi_i}*=\eta^{i\bar{j}}\bar{\xi}_j, && *^{-1}\frac{\partial}{\partial \bar{\xi}_i}*=\eta^{\bar{i}j}\xi_j, && *^{-1}\xi_i* =\eta_{\bar{j}i} \frac{\partial}{\partial \bar{\xi}_j}, && *^{-1}\bar{\xi}_i* =\eta_{j\bar{i}} \frac{\partial}{\partial \xi_j},
\end{align*}
and analogously for the dual variables one can write the expressions for the generators
$\{E_i,F_i,H_i\}$ in the form.
\begin{align*}
E_0&=\xi_i\otimes \xi^i,&
E_1&=\bar{\xi_i}\otimes \eta^{j\bar{i}}\frac{\partial}{\partial \xi^j}&
E_2&=\frac{\partial}{\partial \bar{\xi_i}}\otimes \frac{\partial}{\partial \bar{\xi^i}}&
E_3&=\eta_{j\bar{i}}\frac{\partial}{\partial \xi_j}\otimes \bar{\xi}^i,\\
F_0&=\frac{\partial}{\partial \xi_i}\otimes \frac{\partial}{\partial \xi_i},&
F_1&=\eta_{i\bar{j}}\frac{\partial}{\partial\bar{\xi}_j}\otimes \xi^i,&
F_2&=\bar{\xi}_i\otimes \bar{\xi}^i,&
F_3&=\xi_i\otimes \eta^{i\bar{j}}\frac{\partial}{\partial \bar{\xi}^j}.
\end{align*}
The relations (\ref{EFH}), (\ref{EF}) and (\ref{EE0}) are obvious from the expressions of the generators in coordinates. To prove (\ref{Serre}) compute the commutators:
\begin{align*}
[E_0,E_1]&=\eta^{i\bar{j}}\,\xi_i\bar{\xi}_j\otimes 1,& [E_1,E_2]&=1\otimes\eta^{i\bar{j}}\,\frac{\partial^2}{\partial \xi^i\partial \bar{\xi}^j},\\
[E_2,E_3]&=-\eta_{i\bar{j}}\,\frac{\partial^2}{\partial \xi_i\partial \bar{\xi}_j}\otimes 1,&
[E_3,E_0]&=-1\otimes\eta_{i\bar{j}}\,\xi^i\bar{\xi}^j,\\
[F_0,F_1]&=\eta_{i\bar{j}}\frac{\partial^2}{\partial \xi_i\partial \bar{\xi}_j}\otimes 1,&
[F_1,F_2]&=1\otimes \eta_{i\bar{j}}\,\xi^i\bar{\xi}^j,\\
[F_2,F_3]&=-\eta^{i\bar{j}}\,\xi_i\bar{\xi}_j \otimes 1,&
[F_3,F_0]&=-1\otimes \eta^{i\bar{j}}\,\frac{\partial^2}{\partial \xi^i\partial \bar{\xi}^j},\\
\end{align*}
\begin{align*}
[E_0,F_0]&=\xi_i\frac{\partial}{\partial \xi_i}\otimes 1+1\otimes \xi_i\frac{\partial}{\partial \xi_i}-n=p'+q'-n\\
[E_1,F_1]&=\bar{\xi}_i\frac{\partial}{\partial \bar{\xi}_i}\otimes 1 - 1\otimes \xi_i\frac{\partial}{\partial \xi_i}=p''-q'\\
[E_2,F_2]&=n-\bar{\xi}_i\frac{\partial}{\partial \bar{\xi}_i}\otimes 1+1\otimes \bar{\xi}_i\frac{\partial}{\partial \bar{\xi}_i}=n-p''-q''\\
[E_3,F_3]&=1\otimes \bar{\xi}_i\frac{\partial}{\partial \bar{\xi}_i} -  \xi_i\frac{\partial}{\partial \xi_i}\otimes 1=q''-p'
\end{align*}
The relations (\ref{Serre}) follow. 
\end{proof}

Observe that the embedding $\mathfrak{sl}_2\times \mathfrak{sl}_2\to  \widehat{\mathfrak{sl}}_4$ is given by $E=E_0+F_2$, $E'=E_1+F_3$, $F=F_0+E_2$, $F'=F_1+E_3$.

\subsection{K\"ahler condition.}
The Lie algebra of the unitary group (the group preserving both the complex structure $J$ and the scalar product $\eta$) is embedded into the Clifford algebra as $\cl^{11}(\mathrm{V}^\mathbb{C})$. The connection preserving unitary structure is therefore a form $\omega^{\mathbb{R}}=\omega+\bar{\omega}$ with $\omega\in \Omega^{1110}$ and $\bar{\omega}\in \Omega^{1101}$. 

Observe that $w=\theta\bar{\theta}-\bar{\theta}\theta\in \Omega^{0011}$ is just an ordinary (1,1)-form. 
\begin{proposition}
A torsion-free connection preserving the metric and the complex structure exists if and only if the form $w$ is closed. In this case such connection is unique.
\end{proposition}
\begin{proof}
The torsion $t^{\mathbb{R}}=(\partial+\bar{\partial})(\omega+\bar{\omega})+(\theta+\bar{\theta})(\omega+\bar{\omega})+(\omega+\bar{\omega})(\theta+\bar{\theta})$ decomposes into four components 
\begin{align*}
t^{1020}&=\partial\theta+\omega\theta+\theta\omega= \partial\theta+2F_1\omega, &
t^{0111}&=\partial\bar{\theta}+\omega\bar{\theta}+\bar{\theta}\omega= \partial\bar{\theta}+ 2E_3 \omega,\\
t^{0102}&=\bar{\partial}\bar{\theta}+\bar{\omega}\bar{\theta}+\bar{\theta}\bar{\omega}= \bar{\partial}\bar{\theta}+2E_3\bar{\omega}, &
t^{1011}&=\bar{\partial}\theta+\bar{\omega}\theta+\theta\bar{\omega}= \bar{\partial}\theta+2F_1\bar{\omega}, 
\end{align*}
and vanishing of $t^{\mathbb{R}}$  is equivalent to vanishing of each of these components. It implies that
$$0=E_3t^{1020}-F_1t^{0111}=E_3\partial\theta-2E_3F_1\omega-F_1\partial\bar{\theta} -2F_1E_3 \omega=\partial w/2.$$
The last equality is a consequence of
$$ \partial w=(\partial\theta)\bar{\theta}+\bar{\theta}(\partial\theta)-\theta(\partial\bar{\theta})-(\partial\bar{\theta})\theta=2E_3 \partial\theta-2F_1 \partial\bar{\theta}$$
and of the relation $[F_1E_3]=0$. The proof that $\bar{\partial}w=0$ is analogous and thus the form $w$ is necessarily closed.

Conversely, the condition $dw=0$, the relation $t^{0111}=0$ and the invertibility of $E_3:\Omega^{1110}\to \Omega^{0111}$ imply that the condition $t^{1020}=0$ is also satisfied. Indeed the condition $t^{0111}=0$ implies $\omega= -\frac12 E_3^{-1}\partial\bar{\theta}$. Substituting it into the expression for $t^{1020}$ one gets
$$t^{1020}= \partial\theta-F_1E_3^{-1}\partial\bar{\theta}=E_3^{-1}(E_3\partial\theta-F_1\partial\bar{\theta})=\frac12E_3^{-1}\partial w=0$$.
\end{proof} 

\subsection{Bianchi identity in the K\"ahler case.} 
In the Kahler case  the first Bianchi identity implies that only the component $R\in \Omega^{1111}$ of the curvature $R^\mathbb{R}$ is nonzero and that it is invariant under a subalgebra $\widehat{\mathfrak{sl}}_2\times \widehat{\mathfrak{sl}}_2$ of $\widehat{\mathfrak{sl}}_4$.
\begin{proposition}
The curvature is given by $R^{\mathbb{R}}=R=\partial\bar{\omega}+\bar{\partial}\omega+\omega\bar{\omega}+\omega\bar{\omega}\in \Omega^{1111}$ It is invariant under the subalgebra or $\hat{\mathfrak{sl}}_4$ generated by $E_3,F_3,H_3$ and $E_1,F_1,H_1$.
\end{proposition}
\begin{proof} The curvature $R^\mathbb{R}=(\partial+\bar{\partial})(\omega+\bar{\omega}) + (\omega+\bar{\omega})^2$ decomposes into three components $(\partial\omega+\omega^2)\in \Omega^{1120}$, $\bar{\partial}\bar{\omega}+\bar{\omega}^2\in \Omega^{1102}$ and $\partial\bar{\omega}+\bar{\partial}\omega+\omega\bar{\omega}+\omega\bar{\omega}\in \Omega^{1111}$.
Prove first that $\partial\omega+\omega^2=0$. Indeed
\begin{equation*}
\begin{split}
0=&\partial t^{0111}= \partial\omega\bar{\theta}-\bar{\theta}\partial\omega-\omega\partial\bar{\theta}+\partial\bar{\theta}\omega=\\ =& (\partial \omega +\omega^2)\bar{\theta}-\bar{\theta}(\partial\omega+\omega^2)=-2E_3 (\partial\omega+\omega^2)
\end{split}
\end{equation*}
On the other hand since $H_3|_{\Omega^{1120}}=-1$ the operator $E_3$ is bijective in this degree and therefore $\partial\omega+\omega^2=0$. The identity $\bar{\partial}\bar{\omega}+\bar{\omega}^2=0$ can be proven analogously.

Similarly $0=\bar{\partial}t^{1020}+\partial t^{1011}=R\theta-\theta R=-2E_3R$. On the other hand $H_3R=0$ since $H_3|_{\Omega^{1111}}=0$ and therefore $F_3R=0$. Invariance under $E_1,F_1,H_1$ can be proven analogously.
\end{proof}

\subsection{Einstein equation in the K\"ahler case.}
In the K\"ahler case the Einstein equation is equivalent to the equations $F_0R=0$ and $E_2R=0$. In four real dimension it is equivalent to the full $\widehat{\mathfrak{sl}}_4$ invariance of the curvature. Indeed, the Einstein equation reads as $FR^\mathbb{R}=0$. In the complex case it amounts to $(F_0+E_2)R=0$. Since the two terms have different gradings this equation is equivalent to $F_0R=0$ and $E_2R=0$. Since in dimension four $H_0R=H_2R=0$ it implies together with the Bianchi identities that all generators of $\widehat{\mathfrak{sl}}_4$ annihilate the curvature.

\section{Conclusion.}
First of all we hope that the developed formalism allows to simplify learning Riemann differential geometry for students as well as to better understand the logic of the theory. Emergence of the affine group $\widehat{\mathfrak{sl}}_4$ seems mysterious for us and requires better understanding of its consequences. In the forthcoming version of the paper we plan to apply the construction to more elaborated structures such as hyperk\"ahler, $G_2$ and to supergravity. The main idea of the paper is presented in the master thesis of the second author defended at the University of Strasbourg 19/6/2017.

\begin{thebibliography}{9}
\bibitem{P} A.Palatini, \textit{Deduzione invariantiva delle equazioni gravitazionali dal principio di Hamilton}, Rend.Circ.Mat. Palermo 43, 203 (1919).
\bibitem{C} E.Cartan, \textit{Sur une g\'en\'eralisation de la notion de courbure de Riemann et les espaces \`a torsion}, C. R. Acad. Sci. 174, 593-595 (1922).
\bibitem{W} E. Witten, \textit{(2+1)-Dimensional Gravity as an Exactly Soluble System,} Nucl. Phys.
 B311 (1988) 46.
\bibitem{AHS} M.F.Atiyah, N.J.Hitchin, I.M.Singer: \textit{Self-duality in four-dimensional Riemannian
geometry}, Proc.Roy.Soc.Lond.A 362 (1978), 425–461.
\bibitem{L} A.Losev, \textit{From Berezin integral to Batalin-Vilkovisky formalism. A mathematical physicist's point of view.} in M.Shifman (ed.) \textit{Felix Berezin. Life and Death of the Mastermind of Supermathematics}. World Scientific 2007.
\bibitem{S} P.\v{S}evera, \textit{A remark on the symmetries of the Riemann curvature tensor.}, \texttt{ arxiv:math/0204191}.
\bibitem{DNF} B.A.Dubrovin, A.T.Fomenko, and S.P.Novikov, \textit Modern Geometry - Methods and Applications : Part I, 2nd Edition, Springer-Verlag, N.Y., 1992.
\end{thebibliography}
\end{document}